\documentclass[reqno]{amsart}
\usepackage{import}
\usepackage[T1]{fontenc}
\usepackage[utf8]{inputenc}
\usepackage{lmodern}
\usepackage{url}

\usepackage[numbers]{natbib} 
\usepackage[english]{babel}

\usepackage{amsmath, amsfonts, amssymb, graphicx,amsthm, mathrsfs}
\usepackage{bbm}
\usepackage{amsfonts,amssymb,graphicx, tikz, mathabx, epsfig, listings}
\usepackage{tikz-cd}
\usepackage{hyperref}
\usepackage{caption}

\graphicspath{{pictures/}}

\newtheorem{thm}{Theorem}[section]
\newtheorem{rem}{Remark}[section]

\newtheorem{lem}{Lemma}[section]

\newtheorem{cor}{Corollary}[section]

\newtheorem{question}{Question}[section]

\newtheorem*{prop*}{Proposition}
\newtheorem*{thm*}{Theorem}

\newtheorem*{rem*}{Remark}
\newtheorem*{imprem*}{Important Remark}
\newtheorem*{lem*}{Lemma}
\newtheorem*{dfn*}{Definition}
\newtheorem*{cor*}{Corollary}
\newtheorem*{probs*}{Problems}
\newtheorem*{prob*}{Problem}
\newtheorem*{problem*}{Problem}
\newtheorem*{ex*}{Example}
\newtheorem*{conj*}{Conjecture}
\newtheorem*{state*}{Statement}
\newtheorem*{question*}{Question}

\def\NN{\mathbb{N}}
\def\RR{\mathbb{R}}
\def\QQ{\mathbb{Q}}

\def\NN{\mathbb{N}}
\def\CC{\mathbb{C}}

\def\ep{\varepsilon}

\def\cO{\mathcal{O}}

\definecolor{vio}{RGB}{118, 120, 238}

\definecolor{identifiercolor}{rgb}{.4,.6,.56}
\definecolor{stringcolor}{gray}{0.5}
\definecolor{inactivecolor}{rgb}{0.15,0.15,0.5}

\begin{document}

\title[On Lenstra's criterion for norm-Euclideanity of number fields]{On Lenstra's criterion for norm-Euclideanity of number fields and properties of Dedekind zeta-functions}

\author{Jordan Pertile and Valeriia V. Starichkova}
\address{School of Science, The University of New South Wales, Canberra, Australia}
\email{j.pertile@unsw.edu.au}
\address{School of Science, The University of New South Wales, Canberra, Australia}
\email{v.starichkova@unsw.edu.au}
\date\today
\thanks{Supported by ARC Discovery Project DP240100186 and the Australian Mathematical Society Lift-off Fellowship of the second author.}
\keywords{Norm-Euclidean number fields, bounds for Dedekind zeta-functions, sphere packing, Rogers' constant}

\begin{abstract}
In 1977, Lenstra provided a criterion for norm-Euclideanity of number fields and noted that this criterion becomes ineffective for number fields of large enough degrees under the Generalised Riemann Hypothesis (GRH) for the Dedekind zeta-functions. In the first part of the paper we make Lenstra’s observation explicit by proving that, under GRH, the criterion becomes ineffective for all number fields of degree $n \geq 62238$. This follows from combining the criterion assumption with the explicit lower bound for the discriminant of $K$ under GRH, and the (trivial) upper bound for the minimal proper ideal norm in $\cO_K$.

Unconditionally, the lower bound for the discriminant is too weak to lead to such a contradiction. However, we show that GRH can be replaced by another condition on the Dedekind zeta functions $\zeta_K$, a conjectural lower bound for $\zeta_K$ at a point to the right of $s = 1$. Combined with Zimmert's approach, this condition implies a different type of upper bound for the minimal proper ideal norm and again contradicts Lenstra's criterion for all $n$ large enough. The advantage of the new potential condition on $\zeta_K$ is that it can be computationally checked for number fields of not too large degrees.
\end{abstract}

\maketitle
\section{Introduction} \label{sec: intro}
A commutative ring $R$ is called \emph{Euclidean} if there exists a function $f: R \to \NN \cup \{0\}$ such that:
\begin{itemize}
    \item $f(x) = 0 \Leftrightarrow x = 0$,
    \item for all $a, b \in R$ there exist $q, r \in R$ such that $a = bq + r$ and $f(r) < f(b)$.
\end{itemize}

The Euclideanity of a ring is an important property implying other useful properties, such as being a principal ideal domain (PID) and a unique factorisation domain. In particular, the Euclideanity of algebraic rings of integers remains an important topic with many open questions, see the survey \cite{Lemmermeyer1995} and its updated version from 2004 \cite{Lemmermeyer2004} for the main progress and conjectures in the area. In the case of number fields, there is a natural choice of the function $f$, against which the Euclideanity of $\cO_K$ may be tested --- the norm of $K$ (see Section \ref{subsec: Lenstra Criterion - Setting} for details). We will call $K$ and $\cO_K$ \emph{norm-Euclidean} if $\cO_K$ is Euclidean with respect to the norm of $K$. We recall several significant results in the area:
\begin{enumerate}
    \item There are finitely many quadratic norm-Euclidean number fields which are fully classified. Moreover, all Euclidean quadratic complex number fields are also norm-Euclidean, but this is not true for the real number fields \cite{Harper2004}.
    \item The cyclotomic Euclidean number fields are fully classified \cite{Harper2004}. Most of them are proven to be also norm-Euclidean, however, Lenstra \cite{Lenstra1974} showed that $\QQ[e^{2\pi i/32}]$ is Euclidean and not norm-Euclidean.
    \item Under the Generalised Riemann Hypothesis (GRH), Weinberger \cite{Weinberger1973} showed that $\cO_K$ with an infinite group of units is a PID if and only if $\cO_K$ is Euclidean; in particular, real quadratic number fields are PIDs if and only if they are Euclidean under GRH. Combining this result with the Cohen--Lenstra heuristics \cite{ConehLenstraHeuristics} suggests that there should be infinitely many real quadratic Euclidean number fields.
    \item Heilbronn \cite{Heilbronn1951} proved that for every prime $\ell$, there are only finitely many norm-Euclidean Galois fields of degree $\ell$, and McGown \cite{McGown2012} provided the upper bounds for discriminants of such fields with $\ell > 2$.
\end{enumerate}

In \cite{Lenstra1977Euclidean}, Lenstra introduced a criterion for the norm-Euclideanity of number fields. We will state the criterion in its general form in Section \ref{subsec: Lenstra Criterion - Setting}, but cite one of its corollaries below.

\begin{thm}{\cite[(1.11)]{Lenstra1977Euclidean}} \label{thm: Lenstra-criterion-sigma-n}
Let $K$ be a number field of degree $n$ with discriminant $\Delta$, and let $M$ be the number of ``exceptional units'' (see \S \ref{subsec: Lenstra Criterion - Setting}). Denote by $\sigma_n$ the Rogers constant (\cite{Rogers1958} and \S \ref{sec: sigma-n}) and $\Gamma$ denote the Gamma function. If
\begin{equation} \label{eq: Lenstra-particular-sphere}
    M > \sigma_n \cdot \Gamma\left( 1 + \frac{n}{2} \right) \cdot \left( \frac{4}{\pi n} \right)^{n/2} \cdot \sqrt{|\Delta|},
\end{equation}
then $K$ is norm-Euclidean.
\end{thm}

In Section \ref{subsec: Lenstra Criterion - choice of U} we show that under GRH, \eqref{eq: Lenstra-particular-sphere} implies that $n$ is bounded. Since \eqref{eq: Lenstra-particular-sphere} also implies an upper bound for the discriminant, only finitely many number fields may satisfy \eqref{eq: Lenstra-particular-sphere} by the Hermite–Minkowski theorem. This agrees with the following idea: if one expects that there are only finitely many norm-Euclidean number fields, then any criterion for the norm-Euclideanity should be effective only for finitely many number fields.

We provide an explicit bound for $n$ satisfying \eqref{eq: Lenstra-particular-sphere}.

\begin{thm} \label{thm: main}
Assume GRH. Then \eqref{eq: Lenstra-particular-sphere} fails for all number fields $K$ of degree $n \geq 62 238$.
\end{thm}

Theorem \ref{thm: main} follows from combining \eqref{eq: Lenstra-particular-sphere} with a lower bound for discriminants of number fields under GRH (see \eqref{eq: Odlyzko-GRH-lower} below) and the upper bound $M \leq 2^n$ \cite[(2.2)]{Lenstra1977Euclidean}. The unconditional lower bound for discriminants is too weak to get such a contradiction unless the sharper bounds for $M$ are provided (see Remark \ref{rem: uncond}).

The bound $M \leq 2^n$ was obtained in \cite{Lenstra1977Euclidean} by using that $M \leq \min_{(0) \subsetneq I \subsetneq \cO_K} N(I)$ where the minimum is taken over all non-zero ideals $I \subsetneq \cO_K$. Using the argument by Zimmert \cite{Zimmert1981}, We prove the following upper bound on $\min_{(0) \subsetneq I \subsetneq \cO_K} N(I)$.

\begin{lem} \label{lem:lower-M-zeta-K}
Let $\zeta_K(s) = \sum_{I \subset \cO_K} \frac{1}{N(I)^s}$, $s > 1$, denote the Dedekind zeta-function corresponding to the number field $K$ with the degree $n$ and the discriminant $\Delta$. Let $r$ be the number of real embeddings of $K$ into $\CC$ and $s$ denote the number of pairs of complex conjugate embeddings of $K$. Let $\beta := \beta(n)$ be a positive function tending to $0$ as $n \to \infty$. Then
\begin{align*}
\log \left( \min_{(0) \subsetneq I \subsetneq \cO_K} N(I) \right) \cdot \left(1 - \frac{1}{\zeta_K(1 + \beta(n))}\right) \leq& -\frac{r}{2}(\log (4\pi) + 1 + \gamma) - s(\log (4\pi) + \gamma) \\&+ \log \sqrt{|\Delta|}
+ O(n\beta(n)) + O\left( \frac{1}{\beta(n)}\right) \quad \text{as} \quad n \to \infty,
\end{align*}
where the minimum is taken over all non-zero ideals $I \subsetneq \cO_K$.
\end{lem}

Using Lemma \ref{lem:lower-M-zeta-K}, we can replace GRH in Theorem \ref{thm: main} by an assumption on the lower bounds for Dedekind zeta functions.


\begin{thm} \label{thm: zeta-Lenstra}
Let $(K_m)_{m \in \NN}$ be a sequence of number fields such that $n_m := \deg K_m \to \infty$ as $m \to \infty$.
Assume there exists $0 < \ep < 1$ such that
\begin{equation} \label{eq: zeta-K-lower-uniform}
\liminf_{m \to \infty} \zeta_{K_m}\left( 1 + n_m^{-\ep} \right) > \frac{2\log 2}{2 \log 2 + \gamma - 1} = 1.43879\ldots.
\end{equation}
Then \eqref{eq: Lenstra-particular-sphere} fails for almost all number fields in the sequence $(K_m)_{m \in \NN}$.
\end{thm}

To our knowledge, the property \eqref{eq: zeta-K-lower-uniform} is not explored in the existing literature. The advantage of assuming \eqref{eq: zeta-K-lower-uniform} over assuming GRH is that we can compute $\zeta_{K}\left( 1 + n^{-\ep} \right)$ for any $0 < \varepsilon < 1$ and for a number field $K$ of not too large degree $n$, whereas GRH is not known for any Dedekind zeta function.

In Appendix \ref{subsec: app-plots}, we compute $\zeta_{K_m}\left( 1 + n_m^{-3/4} \right)$ for the cyclotomic number fields $K_m = \QQ[e^{2\pi i/m}]$ for $1 \leq m \leq 350$: the case $\ep = \frac{3}{4}$ suggests a strong increasing pattern for $\zeta_{K_m}\left( 1 + n_m^{-3/4} \right)$ which quickly starts exceeding the right-hand side of \eqref{eq: zeta-K-lower-uniform}. We note that as we mentioned above, the finitude of norm-Euclidean cyclotomic number fields is already proven, and thus implies the failure of \eqref{eq: Lenstra-particular-sphere} for almost all cyclotomic number fields. Nevertheless, the property \eqref{eq: Lenstra-particular-sphere} or, in general, the uniform lower bounds in the families of the Dedekind zeta functions on the right of $1$ may be interesting to explore further.


We also provide a generalisation of Theorem \ref{thm: zeta-Lenstra} in Lemma \ref{lem: gen-zeta-Lenstra}: the right-hand sides of both \eqref{eq: Lenstra-particular-sphere} and \eqref{eq: zeta-K-lower-uniform} depend on the Rogers upper bound for sphere packings (see Section \ref{subsec: Lenstra Criterion - choice of U}). Meanwhile, there are other bounds improving upon the Rogers result --- using them in Lenstra's criterion will improve upon the criterion \eqref{eq: Lenstra-particular-sphere}. We raise some particular open questions in Appendix \ref{subsec: app-plots}.

The outline of the paper is as follows. In Section \ref{sec: Lenstra Criterion}, we introduce Lenstra's criterion for the norm-Euclideanity and its two particular forms which Lenstra used to derive new norm-Euclidean number fields. One of these particular forms relates to the aforementioned Rogers constant $\sigma_n$. In Section \ref{sec: sigma-n}, we provide explicit upper and lower bounds for $\sigma_n$ and use them to prove Theorem \ref{thm: main} in Section \ref{sec: main-thm-proof}. In Section \ref{sec: unconditional}, we prove Lemma \ref{lem:lower-M-zeta-K}, Theorem \ref{thm: zeta-Lenstra} and the generalisation of the theorem, Lemma \ref{lem: gen-zeta-Lenstra}.

\section{Lenstra's criterion for norm-Euclideanity} \label{sec: Lenstra Criterion}

\subsection{Setting} \label{subsec: Lenstra Criterion - Setting}

In \cite{Lenstra1977Euclidean}, Lenstra introduced a criterion for the norm-Euclideanity of number fields which combines packing theory, geometry of numbers, and the information connected to \emph{exceptional units} of their rings of integers.

We first recall the properties from geometry of numbers required for the criterion. Let $K$ be a number field of degree $n$, and $N$ be its classic norm, i.e.\ $N(x) = |\det_{\QQ}(~\cdot x)|$. Then $N$ can be computed using the embeddings of $K$ into $\CC$: let $K$ have $r$ real and $s$ complex embeddings so that $n = r + 2s$, and let $\mathcal{M}: K \hookrightarrow \RR^r \times \CC^s$ be its Minkowski embedding. Then the map
\begin{align*}
    &N_{\RR}: \RR^r \times \CC^s \to \RR_{\geq 0},\\
    &N_{\RR}: (x_1, \ldots, x_{r+s}) \mapsto \prod_{j=1}^{r} |x_j| \prod_{j=r+1}^{r+s} |x_j|^2.
\end{align*}
coincides with $N$ when restricted to $K$. Moreover, let $\cO_K$ denote the ring of integers and $\Delta$ be the discriminant of $K$. Then $\mathcal{M}(\cO_K)$ is a lattice in $\RR^r \times \CC^s \simeq \RR^n$ of covolume $\sqrt{|\Delta|}$.

Let us turn to the information about the exceptional units. We consider the sets
$$S_{\omega_1, \ldots, \omega_m} = \{ \omega_1, \ldots, \omega_m ~|~ \omega_i \in \cO_K,~ \omega_i - \omega_j \in \cO_K^* \},$$
and define $M = \max\limits_{S_{\omega_1, \ldots, \omega_m}}\{m\}$. We also note that $M \geq |S_{0,1}| = 2$. Moreover, if we shift all elements of $S_{\omega_1, \ldots, \omega_m}$ by $-\omega_1$ and then divide them by $\omega_2 - \omega_1$, we will get the set $S_{0, 1, \xi_1, \ldots, \xi_{m-2}}$, with $\xi_i, 1 - \xi_i, \xi_i - \xi_j \in \cO_K$ for any $i, j \in \{1, 2, \ldots, m-2\}, i \ne j$. The elements $\xi \in \cO_K^*$ satisfying $1-\xi \in \cO_K^*$ are called \emph{exceptional units}, see \cite{Houriet2007}, \cite{Chowla1961}, or \cite{Nagell1964}.

Finally, we define the packing information (see \cite{Lenstra1977Euclidean} or \cite[p. 13]{ConwaySloane} for more details). As mentioned above, the Minkowski map $\mathcal{M}$ embeds $K$ into the metric space $\RR^n$ with the classic Euclidean distance $d((x_1,\ldots,x_n),(y_1,\ldots,y_n)) = \sqrt{\sum_{i = 1}^n (x_i - y_i)^2}$. Let $U \subset \RR^n$ be a bounded subset with positive Lebesgue measure $\mu(U)$ and $(a_k)_{k \in \NN}$ be a sequence of points in $\RR^n$. Let $C_L \subset \RR^n$ be a cube of length $L$ centred at $0$. The \emph{density constant} of $\{U + a_k\}_{k \in \NN}$ is defined as
$$\rho(U_{(a_k)}) = \limsup_{L \to \infty}\frac{\sum\limits_{U + a_k \cap C_L \ne \emptyset} \mu(U + a_k)}{\mu(C_L)}.$$
We say that $\{U + a_k\}_{k \in \NN}$ is a \emph{packing} of the set $U$, if $(U + a_i) \cap (U + a_j) = \emptyset$ for any $i \ne j$. Then the \emph{packing constant} of $U$ is
$$\delta(U)=\sup_{\substack{\text{packings} \\ \{U + a_k\}_{k \in \NN}}}\rho(U_{(a_k)}),$$
and the \emph{centre-packing constant} of $U$ is
$$\delta^*(U)=\frac{\delta(U)}{\mu(U)}.$$

\begin{thm}\cite[(1.4)]{Lenstra1977Euclidean} \label{thm: Lenstra-criterion}
Let $K$ be a number field and $\cO_K$, $n$, $\Delta$, $N$, and $M$ be defined as at the beginning of the section. Further, let $U \subset \RR^n$ be a bounded Lebesgue measurable set with positive Lebesgue measure, such that
\begin{equation} \label{eq: prop-U}
    N_{\RR}(u - v) < 1 \textit{ for all } u, v \in U.
\end{equation}
Let $\delta^*(U)$ denote the centre packing constant of $U$. If
\begin{equation} \label{eq: Lenstra-main}
    M > \delta^*(U) \cdot \sqrt{|\Delta|},
\end{equation}
then $K$ is norm-Euclidean.
\end{thm}

\subsection{The choice of set $U$} \label{subsec: Lenstra Criterion - choice of U}

In \cite{Lenstra1977Euclidean}, the author introduces two particular choices of $U$, a parallelepiped and a sphere in $\RR^n$. We briefly introduce both choices. Let
\begin{equation*}
    U_1=\Biggl\{ (x_j)_{j=1}^{r+s} \in \mathbb{R}^r \times \mathbb{C}^s \Bigg| \sum_{j=1}^r |x_j|+2\sum_{j=r+1}^{r+s} |x_j|<\frac{n}{2} \Biggr\},
\end{equation*}
then $\mu(U_1)=\frac{n^n}{n!}\left(\frac{\pi}{4}\right)^s$ and $\delta(U_1)\leq 1$, whence
\begin{equation} \label{eq: def-delta1}
    \delta^*(U_1) \leq \frac{n!}{n^n}\left(\frac{4}{\pi}\right)^s =: \delta^*_1(n).
\end{equation}

Let
\begin{equation*}
    U_2= \Biggl\{ (y_j)_{j=1}^n \in \mathbb{R}^n \Bigg| \sum_{j=1}^n y_j^2<\frac{n}{4} \Biggr\},
\end{equation*}
then
\begin{equation*}
    \mu(U_2) = \left(\frac{n}{4}\right)^{n/2} \frac{\pi^{n/2}}{\Gamma\left( 1 + \frac{n}{2} \right)}, \quad \delta(U_2) \leq \sigma_n,
\end{equation*}
where $\Gamma$ is the Gamma function, and the \emph{Rogers constant} $\sigma_n$ will be defined in Section \ref{sec: sigma-n}. Thus
\begin{equation} \label{eq: def-delta2}
    \delta^*(U_2) \leq \sigma_n \left( \frac{4}{\pi n} \right)^{n/2} \Gamma\left( 1 + \frac{n}{2} \right) =: \delta^*_2(n).
\end{equation}

Thus in \cite{Lenstra1977Euclidean}, the author finds new norm-Euclidean number fields by showing either
\begin{equation} \label{eq: Lenstra-delta-1}
    M > \delta^*_1(n) \sqrt{|\Delta|},
\end{equation}
or
\begin{equation} \label{eq: Lenstra-delta-2}
     M > \delta^*_2(n) \sqrt{|\Delta|}.
\end{equation}
We note that Theorem \ref{thm: Lenstra-criterion-sigma-n} is based on checking \eqref{eq: Lenstra-delta-2}. We show in Lemma \ref{lem: sigman-upper} that $\delta^*_2(n) \leq \delta^*_1(n)$ for all $n \geq 56$, implying that there is no need to check \eqref{eq: Lenstra-delta-1} if one can check \eqref{eq: Lenstra-delta-2} in this range. Meanwhile, the asymptotic for $\sigma_n$ \cite{Rogers1958} and the bound $M \leq 2^n$ \cite{Lenstra1977Euclidean} combined with \eqref{eq: Lenstra-delta-2} imply that
\begin{equation} \label{eq: Lenstra-sigma-n-asymptotic}
    |\Delta|^{1/n} \leq 4 \pi e + o(1), ~n \to \infty,
\end{equation}
and therefore contradicts the bound under GRH \cite{PoitouMinorationsDiscrFev1976} by Serre:
\begin{equation} \label{eq: Odlyzko-GRH-lower}
    |\Delta|^{1/n} \geq 8 \pi e^\gamma + o(1), ~n \to \infty,
\end{equation}
where $\gamma$ is the Euler--Mascheroni constant, since $\frac{8 \pi e^{\gamma}}{4 \pi e} = 2 e^{\gamma - 1} \geq 1.31 > 1$. This means that for $n$ large enough, both criteria \eqref{eq: Lenstra-delta-1} and \eqref{eq: Lenstra-delta-2} do not hold and therefore cannot be used for determining norm-Euclideanity of number fields of very large degrees. We will show $n \geq 62238$ implies the failure of \eqref{eq: Lenstra-delta-2} and thus, of \eqref{eq: Lenstra-delta-1}.

\begin{rem} \label{rem: uncond}
\textnormal{The best unconditional lower bound for discriminants \cite[(2.5)]{OdlyzkoSurvey1990}
\begin{equation} \label{eq: disc-lower-uncond}
    |\Delta|^{1/n} \geq (4\pi e^{1 + \gamma})^{r/n} (4\pi e^{\gamma})^{2s/n} + o(1)
\end{equation}
contradicts \eqref{eq: Lenstra-sigma-n-asymptotic} only if $\frac{r}{n} > 1 - \gamma$.
}
\end{rem}





\section{Explicit bounds for $\sigma_n$} \label{sec: sigma-n}

Let us keep the notations from Section \ref{sec: Lenstra Criterion}.
We define the \emph{Rogers constant}, denoted by $\sigma_n$, as follows. Let $S$ be a regular $n$-simplex in $\mathbb{R}^n$ of edge length 2, and let $V_i$, $i = 1,2,\ldots,n+1$, denote its vertices. Define $T(V_i)$ to be the intersection of $S$ with a ball of radius $1$ centred at $V_i$, and let $T = \bigcup_{i = 1}^{n+1} T(V_i)$. Then 
$$\sigma_n = \frac{\text{vol}(T)}{\text{vol}(S)}.$$

The asymptotic for this constant, $\sigma_n \sim \frac{n}{e} \cdot 2^{-n/2}$, was derived in \cite{Rogers1958}, and an approximate upper bound for some particular choices of $n$ was provided in \cite[p. 743]{LeechSloane1971}:
\begin{equation} \label{eq: Leech-computation-sigman}
    \log_2{\sigma_n}\leq \frac{n}{2}\log_2{\left(\frac{n}{4 e}\right)}+\frac{3}{2}\log_2{\frac{e}{\sqrt{\pi}}}+\frac{5.25}{n+2.5} - \log_2 \Gamma\left( 1 + \frac{n}{2}\right).
\end{equation}
We provide an explicit upper bound for $\sigma_n$ for all $n \geq 1$ and an explicit lower bound for all $n \geq 3$; we start with an upper bound.

\begin{lem} \label{lem: sigman-upper}
We have\footnote{We note that $\log_2($eq.\eqref{eq: sigman-upper}$)$ exceeds \eqref{eq: Leech-computation-sigman} by $\frac{3}{2} \log_2 n + \frac{1}{2} - \frac{3}{2} \log_2 e + \frac{5}{4} \log_2 \pi + \frac{13 \log_2 e}{12n} - \frac{5.25}{n+2.5}$.} for all $n \geq 1$,
\begin{equation} \label{eq: sigman-upper}
    \sigma_n \leq \left(\frac{e}{4n}\right)^{n/2} \frac{(n+1)!}{\Gamma\left( 1 + \frac{n}{2}\right)}.
\end{equation}
As a consequence, $\delta^*_2(n) \leq \delta^*_1(n)$ for all $n \geq 56$, where $\delta^*_1(n)$ and $\delta^*_2(n)$ are defined in \eqref{eq: def-delta1} and \eqref{eq: def-delta2} respectively.
\end{lem}
\begin{proof}
Let $\kappa = \sqrt{\frac{n}{2}}$. The upper bound \eqref{eq: sigman-upper} directly follows from \cite[Section 4]{Rogers1958}. There, it was shown that
\begin{equation} \label{eq: sigma-n-expl}
    \frac{2^n n^{n/2} \sqrt{\pi}}{e^{n/2} (n+1)!} \Gamma\left( 1 + \frac{n}{2} \right) \sigma_n = \int_{-\infty}^{\infty} e^{-u^2} \left[ 2 \kappa e^{-iu/\kappa} \int_{0}^{\infty} e^{-y^2 - 2 \kappa y + 2 i u y} dy \right]^n du,
\end{equation}
with the right-hand side bounded from above by $\int_{-\infty}^{\infty} e^{-u^2} du = \sqrt{\pi}$, which implies \eqref{eq: sigman-upper}. Hence by \eqref{eq: def-delta1}, \eqref{eq: def-delta2}, and \eqref{eq: sigman-upper},
\begin{equation*}
\frac{\delta^*_2(n)}{\delta^*_1(n)} \leq (n+1) \left( \frac{e}{\pi} \right)^{n/2} \left( \frac{\pi}{4} \right)^{s} \leq (n+1) \left( \frac{e}{\pi} \right)^{n/2},
\end{equation*}
and the right-hand side does not exceed $1$ for all $n \geq 56$.
\end{proof}

We note that $\delta_2^{*}(n)$ was computed for $n \leq 36$ in \cite[Table 3]{CohnElkies2003} and for some multiples of $12$ in \cite[Table 1]{CohnZhao2014}. These computations suggest that $\delta^*_2(n) \leq \delta^*_1(n)$ for $n \geq 30$. Hence, by Lemma \ref{lem: sigman-upper}, computing $\delta_2^*(n)$ for $37 \leq n \leq 55$ would suffice to check this inequality for all $n \geq 30$.

Let us now provide an explicit lower bound for $\sigma_n$ for all $n \geq 3$. Let $\kappa = \sqrt{\frac{n}{2}} > 1$, then we will estimate the integral on the right-hand side of \eqref{eq: sigma-n-expl} by considering two cases: $|u| \leq \kappa^{\theta}$ and $|u| > \kappa^{\theta}$ where $0 < \theta < \frac{1}{3}$ is a constant to be defined later. We start with the case $|u| \leq \kappa^{\theta}$, see Lemma \ref{lem: lower bound up to k^theta} below. Let us introduce the necessary notations for the lemma:
\begin{align}
&C_{1,\kappa,\theta} = \frac{1}{8} \left( \frac{2\sqrt{\pi}\sqrt{1 + \kappa^{2\theta - 2}}}{e} + \frac{6}{\kappa} \left( 1 + \frac{\sqrt{1+\kappa^{2\theta - 2}}}{e} \right) + \frac{3}{\kappa^{3}} \right) \label{eq: def-C-1-kappa-theta}\\
&C_{2, \kappa, \theta} = \frac{1}{1 - \kappa^{\theta - 1}} \left( \frac{1}{1 - \frac{\kappa^{\theta-1}}{4}} + \frac{5}{2} \right) \label{eq: def-C-2-kappa-theta}\\
&C_{3, \kappa, \theta} = \sqrt{1 + \frac{\kappa^{2\theta-2}}{4}} \label{eq: def-C-3-kappa-theta}\\
&C_{41, \kappa, \theta} = C_{3, \kappa, \theta} + \kappa^{3\theta - 3} C_{2, \kappa, \theta} C_{3, \kappa, \theta} + \frac{\kappa^{\theta-1}}{4} \label{eq: def-C-41-kappa-theta}\\
&C_{42, \kappa, \theta} = C_{2, \kappa, \theta} \left( 1 - \frac{1}{2\kappa^2}\right). \label{eq: def-C-42-kappa-theta}\\
&C_{\kappa,\theta}(u) = C_{1, \kappa, \theta} + C_{41, \kappa, \theta} |u| + C_{42, \kappa, \theta} |u|^3. \label{eq: def-C-kappa-theta}
\end{align}

Let $P_{\kappa, \theta}(u) = C_{\kappa, \theta}(u) - \frac{\kappa}{2} u^2$, then the polynomial $P_{\kappa, \theta}$ satisfies the following properties which we will use in Lemma \ref{lem: lower bound up to k^theta}:
\begin{itemize}
    \item The coefficients of $P_{\kappa, \theta}(u)$ are positive, decreasing in $\kappa$, and increasing in $\theta$ for all $\kappa > 1$ and $0 < \theta < \frac{1}{3}$.
    \item For all $\kappa > 1$, $P_{\kappa, \theta}(0) > 0$. Moreover, $P_{\kappa, \theta}(\kappa^{\theta}) < 0$ for all $\kappa \geq 24$ and all $0 < \theta < \frac{1}{3}$: this follows from $P_{\kappa, \theta}(\kappa^{\theta}) \leq C_{42, 24, \theta}(\kappa^{\theta}) - \frac{\kappa^{1 + 2\theta}}{2}$. Thus, $P_{\kappa, \theta}(u)$ has a unique root in the interval $[0, \kappa^{\theta}]$. This root does not exceed $0.19$ and is of size $O(\kappa^{-1/2})$ as $\kappa \to \infty$.
\end{itemize}

\begin{lem} \label{lem: lower bound up to k^theta}
Let $\kappa = \sqrt{n/2} \geq 24$, $\nu = 2\kappa \left( 1 - \frac{iu}{\kappa} \right)$, and $0 < \theta < \frac{1}{3}$. Assume $C_{\kappa, \theta}$ is defined in \eqref{eq: def-C-kappa-theta} with subsidiary constants defined by \eqref{eq: def-C-1-kappa-theta}, \eqref{eq: def-C-2-kappa-theta}, \eqref{eq: def-C-3-kappa-theta}, \eqref{eq: def-C-41-kappa-theta}, and \eqref{eq: def-C-42-kappa-theta}. Let $U_{\kappa, \theta}$ be the unique solution to the equation
\begin{equation}
    C_{\kappa, \theta}(u) = \frac{\kappa}{2} u^2
\end{equation}
in the interval $[0, \kappa^{\theta}]$. Then
\begin{align} 
    &\int_{-\kappa^{\theta}}^{\kappa^{\theta}} e^{-u^2} \left[ 2 \kappa e^{-iu/\kappa} \int_{0}^{\infty} e^{-y^2 - \nu y} dy \right]^n du \nonumber \\
    &\geq \int_{-\kappa^{\theta}}^{\kappa^{\theta}} e^{-u^2} \left[ 1 - \frac{u^2}{2\kappa^2} \right]^n du - \frac{2 \sqrt{\pi} C_{\kappa,\theta}(\kappa^{\theta})}{\kappa} - \frac{4 U_{\kappa, \theta}C_{\kappa,\theta}(U_{\kappa, \theta})}{\kappa} \left( 1 + \frac{4 C_{\kappa,\theta}(U_{\kappa, \theta})}{\kappa} \right). \label{eq: lower bound up to k^theta}
\end{align}
\end{lem}
\begin{proof}
The inner integral from $0$ to $\infty$ in the first line of \eqref{eq: lower bound up to k^theta} can be expressed using integration by parts three times as follows:
\begin{equation} \label{eq: integr-by-parts}
    \int_{0}^{\infty} e^{-y^2 - \nu y} dy = \frac{1}{\nu} - \frac{2}{\nu^3} + \frac{12}{\nu^5} - \frac{8}{\nu^5} \int_{0}^{\infty} y e^{-y^2 - \nu y} (\nu^2 y^2 + 3 \nu y + 3) dy.
\end{equation}
The absolute value of the integral on the right-hand side of \eqref{eq: integr-by-parts} is less than or equal to
\begin{align*}
    &\int_{0}^{\infty} \left( |\nu| y^2 e^{-y^2} + 3 y e^{-y^2}\right) |\nu| y e^{- \Re(\nu) y} dy + 3 \int_{0}^{\infty} e^{-y^2} y e^{ - \Re(\nu) y} dy \\
    = & \int_{0}^{\infty} \left( |\nu| y^2 e^{-y^2} + 3 y e^{-y^2}\right) \sqrt{1 + \kappa^{2\theta - 2}} \left( 2\kappa y~ e^{- 2 \kappa y} \right) dy + 3 \int_{0}^{\infty} e^{-y^2} y e^{- 2 \kappa y} dy \\
    \leq & \frac{\sqrt{1 + \kappa^{2\theta - 2}}}{e} \int_{0}^{\infty} \left( |\nu| y^2 e^{-y^2} + 3 y e^{-y^2}\right) dy + 3 \int_{0}^{\infty} y e^{- 2 \kappa y} dy \\
    = & \frac{\sqrt{1 + \kappa^{2\theta - 2}}}{e} \left( |\nu| \frac{\sqrt{\pi}}{4} + \frac{3}{2} \right) + \frac{3}{4\kappa^2}.
\end{align*}
The last bound and \eqref{eq: integr-by-parts} imply
\begin{equation} \label{eq: error-internal-integral}
    \left| \int_{0}^{\infty} e^{-y^2 - \nu y} dy - \frac{1}{\nu} + \frac{2}{\nu^3} \right| \leq |\nu|^{-4} \cdot \left( \frac{\sqrt{1 + \kappa^{2\theta - 2}}}{e} \left( 2\sqrt{\pi} + \frac{12}{|\nu|} \right) + \frac{12}{|\nu|} + \frac{6}{|\nu| \kappa^2} \right).
\end{equation}
Let
\begin{equation} \label{def: E-u-kappa}
    E_{1, \kappa, \theta}(u) := 2 \kappa e^{-iu/\kappa} \left( \int_{0}^{\infty} e^{-y^2 - \nu y} dy - \frac{1}{\nu} + \frac{2}{\nu^3} \right),
\end{equation}
then by \eqref{eq: error-internal-integral},
\begin{equation*}
    \left| E_{1, \kappa,\theta}(u) \right| \leq |\nu|^{-3} \cdot \left( \frac{2\sqrt{\pi}\sqrt{1 + \kappa^{2\theta - 2}}}{e} + \frac{12}{|\nu|} \left( 1 + \frac{\sqrt{1+\kappa^{2\theta - 2}}}{e} \right) + \frac{6}{|\nu| \kappa^2} \right)
    \leq \frac{C_{1,\kappa,\theta}}{\kappa^3}.
\end{equation*}
where we used the definition of $\kappa$ and that $|\nu| \geq 2\kappa$ in the last line. Let us derive the asymptotic for the last two terms in \eqref{def: E-u-kappa}, namely let us show that
\begin{equation} \label{eq: multiply-exponent}
    \frac{2 \kappa e^{-iu/\kappa}}{\nu} - \frac{4 \kappa e^{-iu/\kappa}}{\nu^3} = 1 - \frac{u^2}{2\kappa^2} + O(\kappa^{3\theta-1}),
\end{equation}
and provide an explicit bound for the term $O(\kappa^{3\theta-1})$. Indeed,
\begin{align*}
    \frac{2 \kappa e^{-iu/\kappa}}{\nu} - \frac{4 \kappa e^{-iu/\kappa}}{\nu^3} = \frac{e^{-iu/\kappa}}{1 - \frac{iu}{\kappa}} \left( 1 - \frac{1}{2\kappa^2 \left( 1 - \frac{iu}{\kappa} \right)^2} \right).
\end{align*}
First,
\begin{align*}
    \frac{e^{-iu/\kappa}}{1 - \frac{iu}{\kappa}} &= \left( 1 - \frac{iu}{\kappa} - \frac{u^2}{2\kappa^2} + \sum_{m \geq 3} \frac{1}{m!} \left( \frac{iu}{\kappa} \right)^m\right) \left( 1 + \frac{iu}{\kappa} - \frac{u^2}{\kappa^2} + \sum_{m \geq 3} \left( \frac{iu}{\kappa} \right)^m\right)\\
    &= 1 - \frac{u^2}{2\kappa^2} + E_{2, \kappa, \theta}(u),
\end{align*}
with
\begin{align*}
    |E_{2, \kappa, \theta}(u)| \leq &\sum_{m \geq 3} \frac{|u|^m}{\kappa^m} + \frac{|u|}{\kappa} \sum_{m \geq 2} \frac{|u|^m}{\kappa^m} + \frac{|u|^2}{2\kappa^2} \sum_{m \geq 1} \frac{|u|^m}{\kappa^m} \\
    &+ \frac{|u|^3}{\kappa^3} \left( 1 + \frac{|u|}{4 \kappa} + \frac{|u|^2}{4^2\kappa^2} + \ldots \right) \left( 1 + \frac{|u|}{\kappa} + \frac{|u|^2}{\kappa^2} + \ldots \right)
    = C_{2, \kappa, \theta} \frac{|u|^3}{\kappa^3}. 
\end{align*}
Second,
\begin{align*}
    1 - \frac{1}{2\kappa^2 \left( 1 - \frac{iu}{\kappa} \right)^2} &= 1 - \frac{1}{2\kappa^2} \left( 1 + \frac{1}{\left( 1 - \frac{iu}{\kappa} \right)^2} - 1 \right) \\
    &= 1 - \frac{1}{2\kappa^2} - \frac{1}{2\kappa^2} \cdot \frac{2u}{\kappa} \left( i + \frac{u}{2\kappa} \right) \left( 1 - \frac{iu}{\kappa} \right)^{-2}\\
    &= 1 - \frac{1}{2\kappa^2} - \frac{u}{\kappa^3} \left( i + \frac{u}{2\kappa} \right) \left( 1 - \frac{iu}{\kappa} \right)^{-2} \\
    &= 1 - \frac{1}{2\kappa^2} + E_{3, \kappa, \theta}(u),
\end{align*}
where
\begin{equation*}
    |E_{3, \kappa, \theta}(u)| \leq \frac{|u|}{\kappa^3} C_{3, \kappa, \theta}.
\end{equation*}

Hence, the left-hand side of \eqref{eq: multiply-exponent} is equal to
\begin{equation*}
    \left( 1 - \frac{u^2}{2\kappa^2} + E_{2, \theta}(u, \kappa) \right) \left( 1 - \frac{1}{2\kappa^2} + E_{3, \kappa, \theta}(u) \right) = 1 - \frac{u^2 + 1}{2 \kappa^2} + E_{4, \kappa, \theta}(u),
\end{equation*}
where
\begin{align*}
|E_{4, \kappa, \theta}(u)| \leq \frac{|u|}{\kappa^3} C_{41, \kappa, \theta} + \frac{|u|^3}{\kappa^3} C_{42, \kappa, \theta} = O(\kappa^{3\theta-3}) \quad \text{as} \quad \kappa \to \infty.
\end{align*}

From \eqref{eq: integr-by-parts}, \eqref{eq: error-internal-integral}, and \eqref{def: E-u-kappa}, we have
\begin{align*}
    2 \kappa e^{-iu/\kappa} \int_{0}^{\infty} e^{-y^2 - \nu y} dy = 1 - \frac{u^2 + 1}{2 \kappa^2} + E_{1, \kappa, \theta}(u) + E_{4, \kappa, \theta}(u).
\end{align*}
Let us denote
\begin{equation} \label{eq: def-E-kappa-theta}
    E_{\kappa, \theta}(u) := E_{1, \kappa, \theta}(u) + E_{4, \kappa, \theta}(u),
\end{equation}
then
\begin{equation*}
    |E_{\kappa, \theta}(u)| \leq \frac{C_{1, \kappa, \theta}}{\kappa^3} + \frac{|u|}{\kappa^3} C_{41, \kappa, \theta} + \frac{|u|^3}{\kappa^3} C_{42, \kappa, \theta} = \frac{C_{\kappa,\theta}(u)}{\kappa^3}
\end{equation*}
and the left-hand side of \eqref{eq: lower bound up to k^theta} becomes
\begin{align} \label{eq: int_small_u_1}
    &\int_{-\kappa^{\theta}}^{\kappa^{\theta}} e^{-u^2} \left[ 1 - \frac{u^2}{2\kappa^2} + E_{\kappa, \theta}(u) \right]^n du \nonumber\\
    &\geq \int_{-\kappa^{\theta}}^{\kappa^{\theta}} e^{-u^2} \left[ 1 - \frac{u^2}{2\kappa^2} \right]^n du - \int_{-\kappa^{\theta}}^{\kappa^{\theta}} e^{-u^2} \left( \left[ 1 - \frac{u^2}{2\kappa^2} + |E_{\kappa, \theta}(u)| \right]^n - \left[ 1 - \frac{u^2}{2\kappa^2}\right]^n \right) du.
\end{align}
By the mean-value theorem, the expression inside the second integral can be estimated as follows.
\begin{equation}
    \left[ 1 - \frac{u^2}{2\kappa^2} + |E_{\kappa, \theta}(u)| \right]^n - \left[ 1 - \frac{u^2}{2\kappa^2}\right]^n \leq n |E_{\kappa, \theta}(u)| \left[ 1 - \frac{u^2}{2\kappa^2} + |E_{\kappa, \theta}(u)| \right]^{n - 1}. \label{eq: mean-value}
\end{equation}

Assume $|u| \geq U_{\kappa, \theta}$ with $U_{\kappa, \theta}$ defined in the statement of the lemma. Then
\begin{equation*}
    |E_{\kappa, \theta}(u)| \leq \frac{C_{\kappa,\theta}(u)}{\kappa^3} \leq \frac{u^2}{2\kappa^2},
\end{equation*}
so the right-hand side of \eqref{eq: mean-value} does not exceed $n |E_{\kappa, \theta}(u)| \leq \frac{2 C_{\kappa,\theta}(\kappa^{\theta})}{\kappa}$ for $|u| \leq \kappa^{\theta}$.

In case $|u| < U_{\kappa, \theta}$, we have
\begin{align*}
    \int_{|u| < U_{\kappa, \theta}} n e^{-u^2} |E_{\kappa, \theta}(u)| \left[ 1 - \frac{u^2}{2\kappa^2} + |E_{\kappa, \theta}(u)| \right]^{n - 1} du &\leq 2 n U_{\kappa, \theta} \max_{|u| < U_{\kappa, \theta}}|E_{\kappa, \theta}(u)(1 + 2n|E_{\kappa, \theta}(u)|)| \\
    &\leq \frac{4 U_{\kappa, \theta}C_{\kappa,\theta}(U_{\kappa, \theta})}{\kappa} \left( 1 + \frac{4 C_{\kappa,\theta}(U_{\kappa, \theta})}{\kappa} \right),
\end{align*}
where the inequality in the first line holds since $|E_{\kappa, \theta}(u)| \leq \frac{U_{\kappa, \theta}^2}{n} \leq \frac{1}{n}$ where $U_{\kappa, \theta}^2 \leq 1$ follows from the properties of the polynomial $P_{\kappa, \theta}$ listed above Lemma \ref{lem: lower bound up to k^theta}.
Consequently, the last result with \eqref{eq: int_small_u_1} imply
\begin{align*}
    &\int_{-\kappa^{\theta}}^{\kappa^{\theta}} e^{-u^2} \left[ 1 - \frac{u^2}{2\kappa^2} + E \right]^n du \nonumber\\
    &\geq \int_{-\kappa^{\theta}}^{\kappa^{\theta}} e^{-u^2} \left[ 1 - \frac{u^2}{2\kappa^2} \right]^n du - \frac{2 C_{\kappa,\theta}(\kappa^{\theta})}{\kappa} \int_{-\infty}^{\infty} e^{-u^2} du - \frac{4 U_{\kappa, \theta}C_{\kappa,\theta}(U_{\kappa, \theta})}{\kappa} \left( 1 + \frac{4 C_{\kappa,\theta}(U_{\kappa, \theta})}{\kappa} \right) \\
    &\geq \int_{-\kappa^{\theta}}^{\kappa^{\theta}} e^{-u^2} \left[ 1 - \frac{u^2}{2\kappa^2} \right]^n du - \frac{2 \sqrt{\pi} C_{\kappa,\theta}(\kappa^{\theta})}{\kappa} - \frac{4 U_{\kappa, \theta}C_{\kappa,\theta}(U_{\kappa, \theta})}{\kappa} \left( 1 + \frac{4 C_{\kappa,\theta}(U_{\kappa, \theta})}{\kappa} \right),
\end{align*}
where, to get to the last line, we used that $\int_{-\infty}^{\infty} e^{-u^2} du = \sqrt{\pi}$.
\end{proof}

\begin{cor} \label{cor: bound-sigma-n}
Keep the notations from Lemma \ref{lem: lower bound up to k^theta}. Then 
\begin{equation*}
    \frac{2^n n^{n/2} \sqrt{\pi}}{e^{n/2} (n+1)!} \Gamma\left( 1 + \frac{n}{2} \right) \sigma_n \geq f(\kappa, \theta),
\end{equation*}
where
\begin{equation}
    f(\kappa, \theta) = \int_{-\kappa^{\theta}}^{\kappa^{\theta}} e^{-u^2} \left[ 1 - \frac{u^2}{2\kappa^2} \right]^n du - \frac{2 \sqrt{\pi} C_{\kappa,\theta}(\kappa^{\theta})}{\kappa} - \frac{4 U_{\kappa, \theta}C_{\kappa,\theta}(U_{\kappa, \theta})}{\kappa} \left( 1 + \frac{4 C_{\kappa,\theta}(U_{\kappa, \theta})}{\kappa} \right) - 2 e^{-\kappa^{\theta}}.
\end{equation}
\end{cor}
We note that $f(\kappa, \theta)$ increases in $\kappa$ for all $\kappa \geq 24$ and $0 < \theta < \frac{1}{3}$.
\begin{proof}
We combine \eqref{eq: sigma-n-expl}, Lemma \ref{lem: lower bound up to k^theta}, and the bound below for $|u| > \kappa^{\theta}$:
\begin{align*}
    \left| \int_{|u|>\kappa^{\theta}} e^{-u^2} \left[ 2 \kappa e^{-iu/\kappa} \int_{0}^{\infty} e^{-y^2 - 2 \kappa y + 2 i u y} dy \right]^n du \right| &\leq 2 \int_{\kappa^{\theta}}^{\infty} e^{-u^2} \left[ 2 \kappa \int_{0}^{\infty} e^{-2 \kappa y} dy \right]^n du \\
    &\leq 2 \int_{\kappa^{\theta}}^{\infty} e^{-u^2} du \\
    &\leq 2 \int_{\kappa^{\theta}}^{\infty} e^{-u} du = 2 e^{-\kappa^\theta},
\end{align*}
where in the last inequality we used that $u > \kappa^\theta > 1$ implying $e^{-u^2} \leq e^{-u}$.
\end{proof}

\section{Proof of Theorem \ref{thm: main}} \label{sec: main-thm-proof}

In this section, we prove Theorem \ref{thm: main}. To do that, we combine \eqref{eq: Lenstra-delta-2}, the bound for $\sigma_n$ from Corollary \ref{cor: bound-sigma-n}, and the conditional explicit bound for discriminants due to Poitou, cited below.

\begin{lem} \cite[(10)]{PoitouSurLesPetitsDiscrNov1976}
Assume GRH holds. Then
\begin{equation} \label{eq: lower-Poitou-explicit}
\frac{1}{n} \log |\Delta| \geq \gamma + \log 8 \pi + \frac{r}{n} \left( \frac{\pi}{2} - \frac{2 \pi^2}{\log^2 n} \beta(3) \right) - \frac{2 \pi^2}{\log^2 n} \left( \lambda(3) + \frac{8 + \frac{8}{n}}{\log n \left( 1 + \frac{\pi^2}{\log^2 n} \right)^2} \right),
\end{equation}
with
\begin{align}
    &\lambda(3) = \sum_{n \geq 0} \frac{1}{(2n+1)^3}= \frac{7}{8}\zeta(3),\\
    &\beta(3) = \sum_{n \geq 0} \frac{(-1)^n}{(2n+1)^3} = \frac{\pi^3}{32}.
\end{align}
\end{lem}

\begin{proof}[Proof of Theorem \ref{thm: main}]
The inequality \eqref{eq: Lenstra-delta-2} is equivalent to
\begin{equation*}
    \frac{1}{n} \log|\Delta| \leq \log (n \pi) - \frac{2}{n} \log \left( \sigma_n \Gamma \left( 1 + \frac{n}{2} \right) \right),
\end{equation*}
which by Corollary \ref{cor: bound-sigma-n} with $n = 2\kappa^2$ and \cite{Robbins1955} implies
\begin{align*}
    \frac{2}{n} \log \left( \sigma_n \Gamma \left( 1 + \frac{n}{2} \right) \right) &\geq \frac{2}{n} \log f(\kappa, \theta) - 2 \log 2 - \log n - \frac{\log \pi}{n} + 1 + \frac{2}{n} \log ((n+1)!)\\
    &\geq -\log n + 1 - 2 \log 2 + \frac{2}{n} \log f(\kappa, \theta) - \frac{\log \pi}{n} \\ 
    &+ \frac{2}{n} \left( \frac{\log (2 \pi)}{2} + \left( n + \frac{3}{2} \right) \log (n+1) - n - 1 + \frac{1}{12n + 1} \right).
\end{align*}
We could have used \cite{Maria1965} for a better lower bound on $(n+1)!$ but the bound in \cite{Robbins1955} suffices for our purposes.
Hence, we have
\begin{align*}
    \frac{1}{n} \log|\Delta| &< \log (n \pi) + \log n - 1 + 2 \log 2 - \frac{2}{n} \log f(\kappa, \theta) + \frac{\log \pi}{n} \\ 
    &- \frac{\log (2 \pi)}{n} - \frac{2n+3}{n} \log (n+1) + \frac{2n+2}{n} - \frac{2}{n(12n + 1)} \\
    &< 2 \log n + \log \pi + 2 \log 2 - 1 - \frac{2}{n} \log f(\kappa, \theta) - \frac{\log 2}{n} \\&- 2 \log n - 3 \frac{\log n}{n} + 2 + \frac{2}{n} - \frac{2}{n(12n + 1)}\\
    &< \log (4 \pi e) - 3 \frac{\log n}{n} + \frac{2 - \log 2 - 2 \log f(\kappa, \theta)}{n} - \frac{2}{n(12n+1)}.
\end{align*}

By comparing this bound with \eqref{eq: lower-Poitou-explicit}, we get the inequality
\begin{align*}
    \gamma + \log 2 - 1 &< \frac{2 \pi^2}{\log^2 n} \left( \lambda(3) + \frac{8 + \frac{8}{n}}{\log n \left( 1 + \frac{\pi^2}{\log^2 n} \right)^2} \right) - \frac{r}{n} \left( \frac{\pi}{2} - \frac{2 \pi^2}{\log^2 n} \beta(3) \right) \\
    &- \frac{3 \log n}{n} + \frac{2 - \log 2 - 2 \log f(\kappa, \theta)}{n} - \frac{2}{n(12n+1)},
\end{align*}
which no longer holds for $\theta = 0.1$, all $n = 2\kappa^2 \geq 62238$, and all $0 \leq r \leq n$. We approximated $f(\kappa, \theta)$ in SageMath 10.4; in particular, $f(\sqrt{62238/2}, 0.1) \geq 0.484$ leads to the failure of the inequality above.
\end{proof}

\section{Lenstra's criterion and lower bounds for Dedekind zeta functions} \label{sec: unconditional}
In this section, we prove Lemma \ref{lem:lower-M-zeta-K}, Theorem \ref{thm: zeta-Lenstra} and Lemma \ref{lem: gen-zeta-Lenstra}. To do that, we will use that $M \leq \min_{(0) \subsetneq I \subsetneq \cO_K} N(I)$, where $I$ denotes a proper ideal in $\cO_K$, and $N(I) = \#|\cO_K / I|$ is the norm of $I$ (see \cite[(2.2)]{Lenstra1977Euclidean}).

In \cite{Zimmert1981}, Zimmert showed that at least half of the ideal classes in $K$ contain an ideal $I \subset \cO_K$ such that
\begin{equation*}
    N(I) \leq (4\pi e^{1 + \gamma} + o(1))^{-r/2} (4\pi e^{\gamma} + o(1))^{-s} \sqrt{|\Delta|} \quad \text{as} \quad n \to \infty.
\end{equation*}
Zimmert's result does not imply a bound on $M$ since $I$ may coincide with $\cO_K$. However, we can adapt his argument to prove Lemma \ref{lem:lower-M-zeta-K}.

\begin{proof}[Proof of Lemma \ref{lem:lower-M-zeta-K}]
We use \cite[Satz 4 and Korollar 1]{Zimmert1981} with $f = g = \zeta_K$, $a = r+s$, $b = s$, and $A = \sqrt{|\Delta|} \pi^{-n/2}$: for $0 < \beta < \frac{1}{4}$,
\begin{equation} \label{eq: bound-F-r+s-s<zeta}
    \frac{F_{r+s,s}(\beta)}{2} \leq \frac{\zeta'_K}{\zeta_K}(1 + \beta) + \log \sqrt{|\Delta|} - \frac{n}{2} \log \pi,
\end{equation}
where
\begin{equation}
    F_{a,b}(\beta) = a \left( F_1(\beta) + f_1(\beta)\right) + b \left( F_2(\beta) + f_2(\beta)\right) + F_3(\beta),
\end{equation}
with
\begin{align*}
F_1(\beta) =& \sum_{\ell = 1}^{\infty} \left[ \frac{2}{1+2\beta} \left( \frac{\Gamma'}{\Gamma}\left( \frac{2\ell + 3\beta}{2+4\beta} \right) - \frac{\Gamma'}{\Gamma}\left( \frac{2\ell - 1 + \beta}{2+4\beta} \right) \right) - \frac{1}{2\ell - 2 - \beta} - \frac{1}{2\ell - 1 + \beta}\right] \\
f_1(\beta) =& - \frac{1}{2} \left( \frac{\Gamma'}{\Gamma}\left( \frac{1 + \beta}{2} \right) + \frac{\Gamma'}{\Gamma}\left( -\frac{\beta}{2} \right) \right),\\
F_2(\beta) =& \sum_{\ell = 1}^{\infty} \left[ \frac{2}{1+2\beta} \left( \frac{\Gamma'}{\Gamma}\left( \frac{2\ell + 1 + 3\beta}{2+4\beta} \right) - \frac{\Gamma'}{\Gamma}\left( \frac{2\ell + \beta}{2+4\beta} \right) \right) - \frac{1}{2\ell - 1 - \beta} - \frac{1}{2\ell + \beta}\right], \\
f_2(\beta) =& - \frac{1}{2} \left( \frac{\Gamma'}{\Gamma}\left( 1 + \frac{\beta}{2} \right) + \frac{\Gamma'}{\Gamma}\left( \frac{1-\beta}{2} \right) \right)\\
F_3(\beta) =& - \frac{4}{\beta} + \frac{2}{1 + 2\beta} \left( \frac{\Gamma'}{\Gamma}\left( \frac{1 + \beta}{2+4\beta} \right) - \frac{\Gamma'}{\Gamma}\left( \frac{1 + 3\beta}{2+4\beta} \right) + \frac{\Gamma'}{\Gamma}\left( \frac{2 + 5\beta}{2+4\beta} \right) - \frac{\Gamma'}{\Gamma}\left( \frac{2 + 3\beta}{2+4\beta} \right) \right).
\end{align*}

Let us show that the function $F_1(s) + f_1(s)$ is holomorphic in the region $R = \{s \in \CC, |s| < \frac{1}{4}\}$. All terms in the definition of $F_1(s)$ corresponding to $\ell \geq 2$ are holomorphic in $R$. Thus, the sum for $\ell \geq 2$ uniformly converges to a holomorphic function in $R$. The term corresponding to $\ell = 1$ has a single simple pole at $\beta = 0$ which cancels out with the pole of $f_1(s)$. Hence, $F_1(s) + f_1(s)$ is holomorphic in the desired region. Similarly, $F_2(s) + f_2(s)$ is also holomorphic in the same region. The choice of the region $R$ ensures that all denominators and arguments inserted in $\frac{\Gamma'}{\Gamma}$ are isolated from $0$, and could be enlarged if needed, but is sufficient for our purposes.

Hence, $F_1(s) + f_1(s) = F_1(0) + f_1(0) + O(s)$ and $F_2(s) + f_2(s) = F_2(0) + f_2(0) + O(s)$ uniformly in $R$, from where $F_{a,b}(\beta) - F_3(\beta) = a(F_1(0) + f_1(0)) + b(F_2(0) + f_2(0)) + O((a+b)\beta)$ uniformly for all $a,b$ and $|\beta| \leq \frac{1}{4}$. Moreover, $F_3(\beta) = O \left( \frac{1}{\beta}\right)$ for $0 < |\beta| \leq \frac{1}{4}$, from where we get
\begin{align}
    F_{a,b}(\beta) &= a(F_1(0) + f_1(0)) + b(F_2(0) + f_2(0)) + O((a+b)\beta) + O\left( \frac{1}{\beta} \right) \nonumber\\
    &= a (\gamma + \log 4 + 1) + b (\gamma + \log 4 - 1) + O((a+b)\beta) + O\left( \frac{1}{\beta} \right) \quad \text{as} \quad \beta \to 0, \label{eq: estimate-F-a-b}
\end{align}
where the second line follows from the proof of \cite[Korollar 2]{Zimmert1981}. Combining \eqref{eq: estimate-F-a-b} and \eqref{eq: bound-F-r+s-s<zeta} gives
\begin{equation} \label{eq: bound-Zimmert-1}
    \frac{\zeta'_K}{\zeta_K}(1 + \beta) + \log \sqrt{|\Delta|} \geq \frac{r}{2} (\gamma + \log (4\pi) + 1) + s (\gamma + \log (4\pi)) + O(n \beta) + O\left( \frac{1}{\beta}\right).
\end{equation}

Now we proceed as in \cite{Zimmert1981} but with a slight modification: the author uses that $\frac{\zeta'_K}{\zeta_K}(1 + \gamma) \leq -\log \left(\min_{(0) \subsetneq I} N(I)\right)$ for any $\gamma > 0$. We provide a similar bound involving $\min_{(0) \subsetneq I \subsetneq \cO_K} N(I)$,
$$-\zeta'_K(1 + \gamma) = \sum_{(0) \subsetneq I \subsetneq \cO_K} N(I) \frac{\log N(I)}{N(I)^{1 + \gamma}} \geq \log \left(\min_{(0) \subsetneq I \subsetneq \cO_K} N(I)\right) \cdot (\zeta_K(1 + \gamma) - 1),$$
whence
\begin{align*}
    \frac{\zeta'_K}{\zeta_K}(1 + \gamma) \leq -\log \left(\min_{(0) \subsetneq I \subsetneq \cO_K} N(I)\right) \cdot \left(1 - \frac{1}{\zeta_K(1 + \gamma)}\right),
\end{align*}
which, combined with \eqref{eq: bound-Zimmert-1}, concludes the proof of the lemma.
\end{proof}

Now we are ready to prove Theorem \ref{thm: zeta-Lenstra}.
\begin{proof}[Proof of Theorem \ref{thm: zeta-Lenstra}]
Assume $n$ is such that \eqref{eq: Lenstra-delta-2} is true, then
\begin{equation} \label{eq: lower-bound-M}
\log M > \log \sigma_n + n \log 2 - \frac{n}{2} \log \pi - \frac{n}{2} \log n + \log \Gamma \left( 1 + \frac{n}{2} \right) + \log \sqrt{|\Delta|}.
\end{equation}
As we mentioned before, $\sigma_n \sim \frac{n}{e} \cdot 2^{-n/2}$ \cite{Rogers1958}. By combining this asymptotic with Stirling's formula for $\Gamma$, the equation above implies
\begin{equation*}
    \log M > - \frac{n}{2} (\log \pi + 1) + \frac{3}{2} \log n + \frac{\log \pi}{2} - 1 + \log \sqrt{|\Delta|} + o(1).
\end{equation*}
From Lemma \ref{lem:lower-M-zeta-K} and the observation $M \leq \min_{(0) \subsetneq I \subsetneq \cO_K} N(I)$, we get
\begin{equation} \label{eq: lower-M-zeta-K}
\frac{\log M}{\zeta_K(1 + n_K^{-\ep})} \geq \log M + \frac{r}{2}(\log (4\pi) + 1 + \gamma) + s(\log (4\pi) + \gamma) - \log \sqrt{|\Delta|}
+ O(n^{1-\ep}) + O\left( n^{\ep}\right) \quad \text{as} \quad n \to \infty,
\end{equation}
where we took $\beta(n) = n^{\ep}$ for some $\ep$ satisfying $0 < \ep < 1$.

From \eqref{eq: lower-bound-M} and \eqref{eq: lower-M-zeta-K}, we have
\begin{align*}
    \frac{\log M}{\zeta_K(1 + n^{-\ep})} &> \frac{r}{2}(\log 4 + \gamma) + s(\log 4 + \gamma - 1) + O(n^{1 - \ep}) + o(1)\\
    &\geq \frac{n}{2} (\log 4 + \gamma - 1) + O(n^{1-\ep}) + O(n^{\ep}),
\end{align*}
and thus using $M \leq 2^n$ we obtain
\begin{equation*}
    \zeta_K(1 + n^{-\ep}) < \frac{\log M}{\frac{n}{2}(\log 4 + \gamma - 1) + O(n^{1 - \ep}) + o(1)} \leq \frac{2\log 2}{2\log 2 + \gamma - 1 + O(n^{-\ep}) + O(n^{\ep})},
\end{equation*}
which contradicts \eqref{eq: zeta-K-lower-uniform} for all $n$ large enough.
\end{proof}

As we mentioned in Section \ref{sec: intro}, Theorem \ref{thm: zeta-Lenstra} can be adjusted to other upper bounds for the sphere packings. Namely, the Rogers upper bound $\sigma_n$ satisfies
\begin{equation} \label{eq: cor-zeta-Lenstra-assume}
    \frac{1}{n} \log \sigma_n \sim - 0.5 \cdot \log 2 \quad \text{as} \quad n \to \infty,
\end{equation}
whereas currently the best bound is due to Kabatiansky and Levenshtein \cite{KabLev1978}:
\begin{equation*}
    \frac{1}{n} \log \sigma_n^{\text{KL}} \sim - 0.5990 \cdot \log 2 \quad \text{as} \quad n \to \infty.
\end{equation*}

The statement of Theorem \ref{thm: zeta-Lenstra} can be generalised as follows.

\begin{lem} \label{lem: gen-zeta-Lenstra}
Assume $p_n$, $\frac{1}{n} \log p_n \sim - C$ as $n \to \infty$, is an upper bound for sphere packings used for testing the norm-Euclideanity via Lenstra's criterion
\begin{equation} \label{eq: Lenstra-gen-sphere-packing}
    \log M > p_n \left( \frac{4}{\pi n} \right)^{n/2} \Gamma\left( 1 + \frac{n}{2} \right).
\end{equation}
Assume $2C > 1 - \gamma + 3 \log 2$, and for some $0 < \ep < 1$,
\begin{equation*}
\liminf_{n \to \infty} \zeta_K\left( 1 + n^{-\ep} \right) > \frac{2\log 2}{3 \log 2 + \gamma - 1 - 2C}.
\end{equation*}
Then \eqref{eq: Lenstra-gen-sphere-packing} implies $n$ is bounded.
\end{lem}

We note that the Kabatiansky--Levenshtein bound would still contradict the conditional lower bound \eqref{eq: Odlyzko-GRH-lower} for $n$ large enough. Providing the explicit lower bound for the Kabatiansky--Levenshtein (or the more generalised bound, see \cite{CohnZhao2014} and \cite{CohnElkies2003}) could lead to the analogue of Theorem \ref{thm: main} where \eqref{eq: Lenstra-particular-sphere} would be replaced by a weaker inequality.

\section*{Acknowledgements}

We would like to thank Timothy Trudgian for his suggestion to look into this problem and his continuous support during this project. We would also like to thank Aleksander Simoni\v{c}, Bryce Kerr, John Voight, Felipe Voloch, Brendan Creutz, Steven Galbraith, and Igor Shparlinski for the helpful discussions. We would like to thank Gustav Bagger for computations provided for the cyclotomic number fields which supported the statement of Theorem \ref{thm: zeta-Lenstra}.

\section{Appendix} \label{sec: app}

\subsection{Plots for cyclotomic number fields} \label{subsec: app-plots}

Let $K_m = \QQ[e^{2 \pi i/m}]$ be a cyclotomic number field of degree $n_m = \varphi(m)$. Let $\ep \in \left\{ \frac{1}{4}, \frac{1}{2}, \frac{3}{4} \right\}$. We plot the points $\left(n_m, ~\zeta_{K_m}\left( 1 + n_m^{-\ep}\right) \right)$ with $2 \leq m \leq 350$.

\begin{figure}
    \centering
    \includegraphics[scale = 0.43]{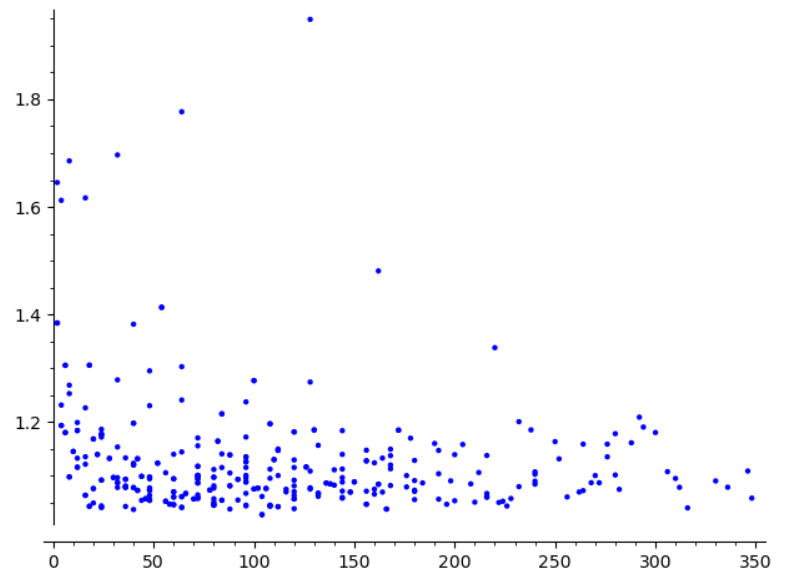}
    \caption{$\ep = \frac{1}{4}$}
    \label{fig:1-4}
\end{figure}

\begin{figure}
    \centering
    \includegraphics[scale = 0.43]{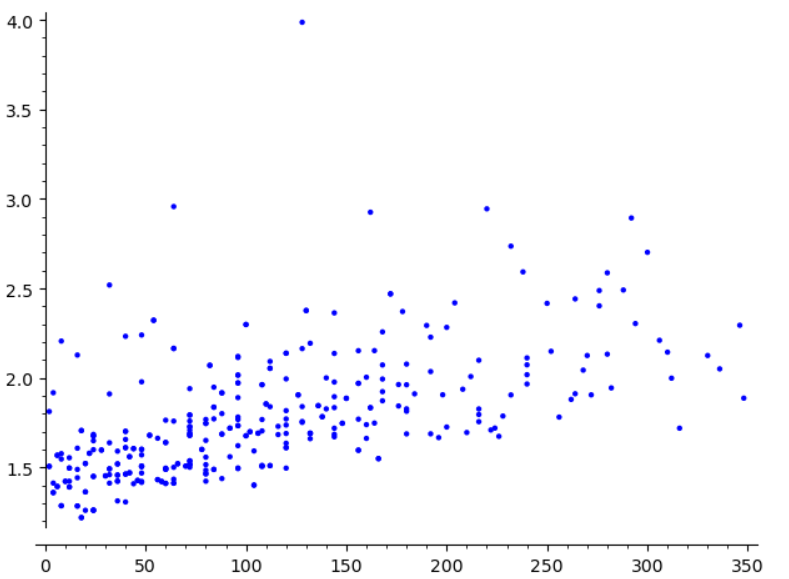}
    \caption{$\ep = \frac{1}{2}$}
    \label{fig:1-2}
\end{figure}

\begin{figure}
    \centering
    \includegraphics[scale = 0.43]{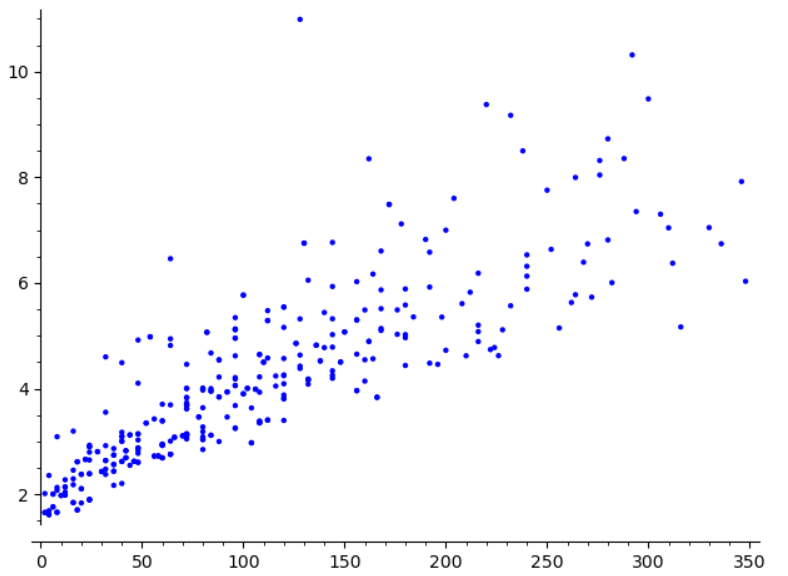}
    \caption{$\ep = \frac{3}{4}$}
    \label{fig:2-4}
\end{figure}

\subsection{Concerning lower bounds for Dedekind zeta functions} \label{subsec: app-questions}

We conclude with the following questions connected to Theorem \ref{thm: zeta-Lenstra} and plots from the previous section.

\begin{question}
\textnormal{Let $(K_m)_{m \in \NN}$ be a sequence of number fields such that $n_m := \deg K_m \to \infty$ as $m \to \infty$. Do there exist constants $0 < \ep < 1$ and $B > 1$ such that
\begin{equation*}
\liminf_{m \to \infty} \zeta_{K_m}\left( 1 + n_m^{-\ep} \right) \geq B?
\end{equation*}}
\end{question}

A tower of number fields may be a particular (and potentially easier) choice of the sequence.

\begin{question}
\textnormal{Let $(K_m)$ be a tower of number fields with $n_m := \deg K_m$. What is the behaviour of $\zeta_{K_m}(s)$ where
\begin{itemize}
    \item $s = 1 + \delta$, with $\delta$ a positive constant;
    \item $s = 1 + n_m^{-\ep}$, with $\ep$ a positive constant;
    \item $s = 1 + g(n_m)$, with $g(t) > 0$, $g(t) \to 0$ as $t \to \infty$?
\end{itemize}}
\end{question}

\bibliographystyle{apalike}
\bibliography{refs.bib}

\end{document}